\documentclass[a4paper,10pt, draft]{article}
\usepackage{cite, amsmath,amsthm,amssymb,graphicx,nicefrac, paralist}

\newtheorem{thm}{Theorem}[section]
\newtheorem{lem}{Lemma}[section]
\newtheorem{pro}{Proposition}[section]

\newtheorem{Def}{Definition}[section]
\newtheorem{obs}{Observation}[section]

\newtheorem{ex}{Example}[section]

\newcommand{\Aff}[1]{\operatorname{Aff}(#1)}
\newcommand{\Proj}[1]{\operatorname{Proj}(#1)}

\title{Vertex Collapsing and Cut Ideals}
\author{Ivan Martino\footnote{Stockholm University, Department of Mathematics, email: martino@math.su.se}}

\begin{document}

\maketitle

\begin{abstract}
In this work we study how some elementary graph operations (like the disjoint union) and the collapse of two vertices modify the cut ideal of a graph. 
They pave the way for reducing the cut ideal of every graph to the cut ideal of smaller ones.

To deal with the collapse operation we generalize the definition of cut ideal given in literature, introducing the concepts of edge labeling and edge multiplicity: in fact we state the \emph{non-classical behavior} of the cut ideal. 
Moreover we show the transformation of the toric map hidden behind these operations. 
\end{abstract}

\section{Introduction}
Recently Sturmfels and Sullivant \cite{Sturmfels-Sullivant} generalized a class of toric ideals which appears in phylogenetics and algebraic statistics \cite{Diaconis-Sturmfels}, via the cut ideals. 
\emph{Cuts} are a key concept in graph theory and combinatorial optimization. 
Monomial cuts ideals have been further studied in \cite{engstrom}, \cite{Nagee-Petro}, \cite{oltreanu}.

Geometrically, the cut ideal of a graph $G$ with $e$ edges comes from the cut polytope $Cut^{\square}(G)$, the convex hull in $\mathbb{R}^{e}$ of the \emph{cut semimetrics} \cite{Deza-Laurent}.

Our first aim is to understand the geometry of the cut varieties.
In particular we study the change of the cut variety under some elementary graph operations (relabeling of vertices and edges, change of multiplicity and disjoint union of graphs) and under the \emph{collapse} of two non-connected vertices.

To understand the collapse operation we consider the graph $P_3$, with vertices $\{1,2,3\}$ and edges $\{\{1,2\},\{2,3\}\}$.
Let us collapse the vertices $1$ and $3$. 
The result is a graph with two vertices and one edge, $K_2$:
This edge should be thought of as a \emph{double} edge. 

In Section \ref{sec-generalized-cut-ideal}, we generalize the concept of cut ideal to a graph $G$ with edge multiplicities and with non trivial edge labels.
In the \emph{classical} case, the multiplicities of all the edges are set to be one and the labeling is the canonical labeling (the label of the edge $\{i,j\}$ is $(i,j)$).  Otherwise, we are in the \emph{non-classical} case.

In Section \ref{sec-elementary-operations}, we explain how to tackle the non-classical setting and we deal with the elementary operations clique $0$-sum \cite{Sturmfels-Sullivant} and disjoint union of two graphs $G\sqcup H$. 

After defining the collapsed graph $G_{1\equiv k}$, in Section \ref{sec-collapsing}, Theorem \ref{thm-collapsing-generic-form} shows how to obtain the cut ideal of $G_{1\equiv k}$ in two steps: $(kill)$, where we delete the generators containing \emph{non-feasible} variables, and $(substitute)$, where we modify the name of the variables according to the collapse.
The collapse operation we work with is different than the clique $i$-sum $G\#_i H$ of $G$ and $H$. It is possible to construct every clique $i$-sum with a finite number of collapse operations.
For this reason, Theorem \ref{thm-collapsing-generic-form} is a generalization of Theorem 2.1 by Sturmfels and Sullivant\cite{Sturmfels-Sullivant}. 

The collapse and the disjoint union operations allow to construct every graph from more elementary graphs: they pave the way for reducing the cut ideal of every graph to the cut ideal of simpler ones.

\paragraph{Notation:} We denote a graph by a pair $G=(V_G;E_G)$, with $V_G=[n]$. 
%
%
Moreover $K_n$, $P_n$, $C_n$ and $S_n$ denote respectively the complete $n$-graph, the $n$-path graph, the $n$-cycle graph and the graph with $n$ isolated vertices graph.

\section{The Generalized Cut Ideal}\label{sec-generalized-cut-ideal}
In this section we generalize the definition of cut ideal.

Let $\Pi_n$ be set of disjoint unordered partitions $A|B$ of $[n]$, that is $A\cup B=\{1,2,\dots, n\}$ and $A\cap B =\emptyset$.
$A|B$ is the same partition as $B|A$ and we denote $(A|B)^*=B|A$. 
We define $Cut_{G}(A|B) = \{\{i, j\} \in E_G : i \in A \textrm{ and } j \in B \text{ or } i \in B \textrm{ and } j \in A\}$. 
Let $R_n$ be $\mathbb{K}[r_{A|B}: A|B \in \Pi_n]$.
\begin{ex}\em
  $R_{4}=\mathbb{K}[r_{1|234}, r_{2|134}, r_{3|124}, r_{4|123}, r_{12|34}, r_{13|24}, r_{14|23}, r_{1234|\cdot}]$.
\end{ex}

A \emph{labeling} of $E$ (or of the graph $G$) is a surjective map $l:E\rightarrow A_G$. 
%

\begin{ex}\em\label{ex-singleton-labeling}
	The graph $K_3$ could be labelled by the map sending all the edges to $\{a\}$.
\end{ex}

There is a particular labeling $c:E\rightarrow E$ that maps an edge $\{i, j\}\in E$ to its endpoints $(i,j)$. This labeling is a one to one map, called the \emph{canonical labeling}. 
All the classical definitions are in the canonical labeling.

Let $T_{A}$ be $\mathbb{K}[s_{a}^{\pm}, t_{a}^{\pm}: a \in A]$. 

\begin{ex}\em
	Let $G$ be any graph with only one edge $E_G=\{\{1,2\}\}$ and with the canonical labeling (so $A_G=\{(1,2)\}$), then $T_{E}=\mathbb{K}[s_{(1,2)}, s_{(1,2)}^{-1}, t_{(1,2)}, t_{(1,2)}^{-1}]$.
\end{ex}

A \emph{multiplicity map} of $E$ (or of the graph $G$) is a map $\sigma:E \rightarrow \mathbb{Z}\setminus\{0\}$. $\sigma(e)$ is called multiplicity of $e$.

\begin{ex}\em
  	Any graph $G$ has the \emph{trivial multiplicity map} setting $\sigma(e)=1$ for each edge $e$.
\end{ex}

\begin{ex}\em\label{ex-multiple-graph-K_2}
	We assign a multiplicity map to $K_2$ setting $\sigma(1,2)\in\mathbb{Z}\setminus\{0\}$.
\end{ex}

We define
\begin{eqnarray*}
  \phi_{(G, l, A,\sigma)}: R_n 		&\rightarrow & 	T_A,\\
	  r_{A|B}&\mapsto&\prod_{\{i, j\}\in Cut_{G}(A|B)} s_{l(i,j)}^{\sigma(i,j)}\prod_{\{i, j\}\in E_G\setminus Cut_{G}(A|B)} t_{l(i,j)}^{\sigma(i,j)}.
\end{eqnarray*}

Roughly speaking, we send the variable $r_{A|B}$ to the product of variables in $T_A$ including $s_{l(i,j)}^{\sigma(i,j)}$ if $A|B$ separates the extremal vertices $i$ and $j$ of the edge $\{i,j\}$ and including $t_{l(i,j)}^{\sigma(i,j)}$ otherwise. 
This explains the names for $s$, separated, and $t$, together. 

If the multiplicity map has value in $\mathbb{N}$, then the map $\phi_{(G, l, A,\sigma)}$ has values in  $T_A'=\mathbb{K}[s_{a}, t_{a}: a \in A]$.

\begin{ex}\em\label{ex-classical-map}
  In the classical case (that is with the canonical labeling and the trivial multiplicity) the map $\phi_{(G, c, E, 1)}$ is 
\begin{eqnarray*}
  \phi_G: R_n 		&\rightarrow & 	T_E',\\
	  r_{A|B}	&\mapsto&	\prod_{\{i, j\}\in Cut(A|B)} s_{(i,j)}\prod_{\{i, j\}\in E_G\setminus Cut(A|B)} t_{(i,j)}.
\end{eqnarray*}
\end{ex}

\begin{obs}
In the classical case the map $\phi_G$ determines $G$ uniquely. 
\end{obs}

\begin{Def}\label{def-cut-ideal}
Let $(G, l,A, \sigma)$ be a labelled graph with multiplicity; The \emph{cut ideal} of $G$, $I_{(G, l,A, \sigma)}$, is the kernel of the map $\phi_{(G,l, A,\sigma)}$. The \emph{affine cut variety} of $G$, $\Aff{G,l, A,\sigma}$, is the affine variety with the coordinate ring $\Gamma_{(G, l,A,\sigma)}=\nicefrac{R_n}{I_{(G, l,A,\sigma)}}$.
\end{Def}
Fixing a grading for $\Gamma_{(G, l,A,\sigma)}$ it could be possible define also the \emph{projective cut variety}, $\Proj{G,l, A,\sigma}$, as the projective variety with the graded coordinate ring $\Gamma_{(G, l,A,\sigma)}$.
In this paper, we avoid it.
\begin{ex}\em\label{ex-classical-and-non-3-seg}
  We study $P_3$ \textbf{a)} in the classical case; \textbf{b)} with trivial multiplicity but with the labeling given in Example \ref{ex-singleton-labeling}; \textbf{c)} with the same labeling but we also fix the multiplicity map $\sigma(1,2)=-\sigma(2,3)=-1$. One has:
  \begin{eqnarray*}
  	\phi_{P_3}:\mathbb{K}[r_{1|23}, r_{2|13}, r_{3|12}, r_{123|\cdot}] &\rightarrow& \mathbb{K}[s_{(1,2)}, s_{(2,3)}, t_{(1,2)}, t_{(2,3)}];\\
	\phi_{(P_3, \{a\})}:\mathbb{K}[r_{1|23}, r_{2|13}, r_{3|12}, r_{123|\cdot}] &\rightarrow& \mathbb{K}[s_{a}, t_{a}].\\
	\phi_{(P_3, \{a\}, \sigma)}:\mathbb{K}[r_{1|23}, r_{2|13}, r_{3|12}, r_{123|\cdot}] &\rightarrow& \mathbb{K}[s_{a}^{\pm 1}, t_{a}^{\pm 1}].
  \end{eqnarray*}
  and 
  \begin{center}
  \begin{tabular}{|c|c|c|c|}\hline
	 	variable & $\phi_{P_3}$ & $\phi_{(P_3, \{a\})}$ 	& $\phi_{(P_3, \{a\}, \sigma)}$\\ \hline
	$r_{1|23}$ 	& $s_{(1,2)}t_{(2,3)}$ & $s_{a}t_{a}$		& $s_{a}^{-1}t_{a}$\\ \hline  
    	$r_{2|13}$ 	& $s_{(1,2)}s_{(2,3)}$ & $s_{a}s_{a} = s_a^2$	& $s_{a}^{-1}s_{a} = 1$\\ \hline
    	$r_{3|12}$ 	& $t_{(1,2)}s_{(2,3)}$ & $t_{a}s_{a}$		& $t_{a}^{-1}s_{a}$\\ \hline
    	$r_{123|\cdot}$ & $t_{(1,2)}t_{(2,3)}$ & $t_{a}t_{a} = t_a^2$ 	& $t_{a}^{-1}t_{a} = 1$\\ \hline
  \end{tabular}
  \end{center}
  Thus, we have $I_{P_3}=(r_{1|23}r_{3|12}-r_{123|\cdot}r_{2|13})$, $I_{(P_3,\{a\})}=(r_{1|23}-r_{3|12}) \oplus I_{P_3}$ and $I_{(P_3,\{a\}, \sigma)}=(r_{2|13}-1, r_{123|\cdot}-1, r_{1|23}r_{3|12} -1)$.
  The classical and the non-classical cut ideals are, hence, different. Looking at the cut varieties we get that $\operatorname{dim}(\Aff{P_3})=3$, $\operatorname{dim}(\Aff{P_3, \{a\}})=2$ and $\operatorname{dim}(\Aff{P_3, \{a\}, \sigma})=1$.
\end{ex}

\section{The Elementary Operations}\label{sec-elementary-operations}
Since $\phi_{(G, l, A, \sigma)}$ is a toric map then the cut varieties are toric varieties, thus following \cite{Sottile_toric}, we associate to $\phi_{(G, l, A, \sigma)}$ the matrix $\mathcal{A}_{(G, l, A, \sigma)}$ having as columns the exponents of the monomial image of $r_{A|B}$ for each partition in $\Pi_n$. 
The generators of $I_{G}$ corresponds to the elements in the kernel of the linear map defined by $\mathcal{A}_{G}$. 

We want to study which correlation there is between elementary operations on the graph and linear transformations of the matrices $\mathcal{A}$. 

\paragraph{Notation:} The \emph{disjoint union} of $G$ and $H$ is $(G, l, A, \sigma)\sqcup (H, i, B, \rho)=(G\sqcup H, (l,i), A\sqcup B,(\sigma , \rho))$.
$G\#_0 H$ is the \emph{clique $0$-sum} \cite{Sturmfels-Sullivant} of $G$ and $H$. 

Any disjoint partition of $G\sqcup H$ can be written as $(AC|BD)$ where $a=(A|B)$ is a $G$-partition and $b=(C|D)$ is a $H$-partition. So we think it as a product partition $a\times b$. From $a$ and $b$ it is also possible to construct $a\times b^*\neq a\times b$.
\begin{lem}\label{lem-elementary-operation}
Let $G$ be a graph in the classical case. 
Let $E=\{e_1, \dots, e_m\}$ and $A=\{a_1, \dots, a_k\}$.
Let $\sigma$ be a multiplicity map and $l:E\rightarrow A$ be a labeling map of $G$.
Then
\begin{compactdesc}
  \item[i)] Permuting the name of vertices corresponds to a permutation of the matrix columns.
  \item[ii)] There exists a unique matrix $M_{\sigma}$ such that $\mathcal{A}_{(G,\sigma)}=M_{\sigma}\mathcal{A}_{G}$. $M_{\sigma}$ is a block matrix
	\[
		M_{\sigma}=\left( \begin{array}{cc}
				 I_{\sigma}	& 0 		\\
				 0		&  I_{\sigma}	\\
	                  \end{array}\right) 
	\]
	where $I_{\sigma}=\operatorname{diag(\sigma(e_1),\dots, \sigma(e_m))}$.
  \item[iii)] There exists a unique matrix $R_{l}$ such that $\mathcal{A}_{(G,l,A)}=R_{l}\mathcal{A}_{G}$. This matrix is $2|A|\times 2|E|$ and it has the block form
	\[
		R_{l}=\left( \begin{array}{cc}
				 B_{l}	& 0 		\\
				 0		&  B_{l}	\\
	                  \end{array}\right) 
	\]
	where $B_l=(b_{i,j})$ is $|A|\times |E|$ and it is defined by $b_{(i,j)}=\delta_{a_i, l(e_j)}$.
\end{compactdesc}
Let $G$ and $H$ be graphs in the non-classical case.
\begin{compactdesc}
  \item[iii.bis)] Let $l'$ be a labeling constructed from $l$ by assigning to the elements in $l^{-1}(a_k)$ a unique element in $C=A\setminus \{a_k\}$. Then there exists a unique matrix $R_{l'}$ such that $\mathcal{A}_{(G,l',C)}=R_{l'}\mathcal{A}_{G}$. This matrix is $2|C|\times 2|A|$ and it has the block form
	\[
		R_{l'}=\left( \begin{array}{cc}
				 B_{l'}	& 0 		\\
				 0		&  B_{l'}	\\
	                  \end{array}\right) 
	\]
	where $B_{l'}=(b_{i,j})$ is $|C|\times |A|$ and it is defined by $b_{(i,j)}=\delta_{a_i, l'(a_j)}$.
 \item[iv)] $\mathcal{A}_{G\sqcup H}$ is made of the columns of $\mathcal{A}_{G} \# \mathcal{A}_{H}$ but each repeated twice.
 \item[v)] $\mathcal{A}_{G\#_0 H}=\mathcal{A}_{G} \# \mathcal{A}_{H}$.
\end{compactdesc}
\end{lem}

\paragraph{Notation:} The columns of $\mathcal{A}_{G} \# \mathcal{A}_{H}$ are constructed mixing the columns of $\mathcal{A}_{G}$ and $\mathcal{A}_{H}$ in all the possible ways.

\begin{proof}
	\textbf{i)}, \textbf{ii)}, \textbf{iii)} and \textbf{iii.bis)} are elementary. 
	Regarding \textbf{iv)}, we observe that for each $a$ and $b$, respectively $G$ and $H$-partitions, the $G\sqcup H$-partitions $a\times b$ and $a\times b^*$ separate and leave together the same edges.
	\textbf{v)} holds because whatever pair of vertices $v\in V_G$ and $w\in V_H$ we choose for the clique $0$-sum $G\#_0 H$, one and only one of those partitions $a\times b$ and $a\times b^*$ leave the pair on one size.
\end{proof}

Using this matrix tricks we get some information about the cut variety:
\begin{thm}\label{teo-elementaty-operation}
Let $G$ be a graph in the classical case.
\begin{compactdesc}
 	\item[0)] $\Aff{S_n}$ is a point for each $n\in \mathbb{N}$.
	\item[1)] $\Aff{G}\cong \Aff{G,\sigma}$ and $I_{G}=I_{(G,\sigma)}$ for every multiplicity map $\sigma$.
\end{compactdesc}
Let $G$ and $H$ be graphs in the non-classical case.
\begin{compactdesc}
	\item[2)] $\Aff{G\sqcup H}\cong \Aff{G\#_0 H}$.
	\item[3)] $\phi_{G\sqcup H}=\phi_G\phi_H$.
	\item[4)] $\Aff{G\sqcup S_n}\cong \Aff{G}$ and $I_{G\sqcup S_n}=I_{G}'\oplus J$, where $J$ is generated only by linear relations.
	\item[5)] An arbitrary binomial lies in $I_{(G\sqcup H)}$ if and only if either it is linear of the form $r_{a\times b}-r_{a\times b^*}$ with $a$ (resp. $b$) disjoint partitions of the graph $G$ (resp. $H$) or it is non linear with the form
		\begin{equation}\label{eq-binomial-form}
  			r_{a_1\times b_1}\cdots r_{a_h\times b_h}-r_{a_{h+1}\times b_{h+1}}\cdots r_{a_{2h}\times b_{2h}},
		\end{equation}
		where
		\begin{equation}\label{eq-cond-1}
  			r_{a_1}\cdots r_{a_h}-r_{a_{h+1}}\cdots r_{a_{2h}}\in I_{(G, l, A, \sigma)},
		\end{equation}
		and
		\begin{equation}\label{eq-cond-2}
  			r_{b_1}\cdots r_{b_h}-r_{b_{h+1}}\cdots r_{b_{2h}}\in I_{(H, i, B, \rho)}.
		\end{equation}
\end{compactdesc}
\end{thm}
\begin{proof}
	$\phi_{S_n}$ sends all variables of $R_n$ (and $1$) to $1\in T_{\emptyset}=\mathbb{K}$; thus \textbf{0)} holds.
	\textbf{1)} follows from \textbf{ii)} and \textbf{2)} follows from \textbf{iv)} and \textbf{v)}. 
	\textbf{3)} is \textbf{iv)} translated with the homomorphism language.
	
	The first part of \textbf{4)} is a consequence of \textbf{2)}. 
	For the latter we observe that for any $S_n$-disjoint partition $(C|D)$, using \textbf{3)}, $r_{(A|BCD)}$ has the same image of $r_{(AC|BD)}$ and of all the other possible further combinations. 
	Thus $J$ is generated by those linear relations and $I_{G}'$ is constructed from $I_{G}$ by replacing the variable $r_{(A|B)}$ with $r_{(A|BCD)}$.
	Using \textbf{3)} and \textbf{iv)}, we obtain \textbf{5)}.
\end{proof}

\begin{obs}
	Example \ref{ex-classical-and-non-3-seg} shows that \textbf{2)} is not true for a non-classical setting.
\end{obs}

\begin{obs}
	\textbf{5)} does not give a method to obtain a minimal base of gene\-rators for the ideal from the two minimal basis of $I_{G}$ and $I_{H}$.
\end{obs}

\begin{ex}\em\label{ex-G-3-SEG}
  One has that $I_{K_2}=(0)$ and
\[
  I_{K_2\sqcup K_2}=\left( \begin{array}{c}
			  r_{1|234}r_{3|124} - r_{24|13} r_{1234|\cdot},\\
			  r_{1234|\cdot}-r_{12|34}, r_{4|123}-r_{3|124},\\
			  r_{2|134}-r_{1|234}, r_{13|24}-r_{14|23}
                        \end{array}\right).
\]
We have no generator in $I_{K_2}$ to construct $I_{K_2\sqcup K_2}$ using \textbf{5)}. Instead, we use binomials like $r_{1|2}-r_{1|2}\in I_{K_2}$.
\end{ex}
%

\section{The Collapse Operation}\label{sec-collapsing}
In this section, we study what happens to the cut ideals and the cut varieties after collapsing two vertices. 
In the first part, we study the \emph{simple} collapse, and then we will go to the \emph{singular} one. 
We see how the non-canonical labeling and non-trivial multiplicity appear naturally.

This is the collapse operation:

\begin{Def}\label{def-simple-collapsing-vertices}
  Let $(G, l, A,\sigma)$ be a graph and let $\{1,k\}\notin E$. 
  The \emph{graph obtained by collapsing the vertices $1$ and $k$} is denoted by $(G_{1\equiv k}, l', A', \sigma')$. 
  We define $G_{1\equiv k}=(\{1,\dots, k-1\};E')$ where $E'$ is obtained from $E_G$ by replacing $\{i,k\}\in E_G$ with $\{i,1\}$, and considering just one repetition; 
  the labeling map $l'$ is the same as $l$, but for all the edges in $e\in l^{-1}(l(i,k))$ we set $l'(e)=l(i,1)$;
  \[
    \sigma'(\{i,j\})=\begin{cases}
			  \sigma(\{i,j\}) 			&\text{ if  $ i\neq 1$;}\\
			  \sigma(\{1,j\})+\sigma(\{k,j\})	&\text{ otherwise.}
		      \end{cases}
  \]
  We say that the collapse is \emph{simple} if $|E_G|=|E_{G_{1\equiv k}}|$, and \emph{singular} otherwise.
\end{Def}

Only for singular collapse we will have that $A'\subsetneq A$: in fact we lose one of the labels of the collapsed edges. 

\begin{obs}
 	When we write $(G\sqcup H)_{k\equiv k+1}$ and $G\#_0 H$ we mean the same thing. We prefer the first if we want to specify the collapsing vertices, the second one if this is clear from the context.
	Moreover, every clique $i$-sum can be constructed as a sequence of collapses, but, of course, the opposite is false.
\end{obs}

\begin{ex}\em\label{ex-no-simple-collapsing}
	$K_2=(P_3)_{1\equiv 3}$. 
	This collapse is singular and it produces the graph with multiplicity given in Example \ref{ex-multiple-graph-K_2}.
\end{ex}


\begin{ex}\em
  The singular collapse can involve more than two edges. For example $G_{1\equiv 5}$, where
  \[
	G=(\{1,2,3,4,5\}, \{\{1,2\},\{1,3\},\{1,4\},\{5,2\},\{5,3\},\{5,4\}\}).
  \]
\end{ex}

\begin{Def}
  A disjoint partition $(A|B)$ is feasible for the collapse of $1$ and $k$ if $\{1,k\}\in Cut_{K_n}(A|B)$. 
\end{Def}
In other words we require that the collapsed vertices belongs both to either $A$ or $B$. 
If $p$ is a feasible partition then $r_p\in R_n$ is a \emph{feasible variable}.
Moreover, let $p$ be a disjoint partition of $[n-1]$, then we denote by $p_f$ and $p_{nf}$ the feasible and the non-feasible lifting to the partitions of $[n]$. 

The following theorem shows a pure combinatorial description of the ideal $I_{G_{1\equiv k}}$ from the ideal of $I_G$. This theorem is stated in the classical and non classical case for simple and singular collapse. 

\begin{thm}[Collapsing rules]\label{thm-collapsing-generic-form}
  Let $G$ be a graph and let $\{1, k\}\notin E$; one obtains $I_{G_{1\equiv k}}$ from $I_G$, using the following rules:
  \begin{compactdesc}
   \item [(kill)] Kill all the elements in $I_G$ having non-feasible variables; 
   \item [(substitute)] Substitute $r_{p_f}$, the variable of $R_G$, with $r_{p}$, variable of $R_{G_{1\equiv k}}$, in the 'surviving' elements of $I_G$.
  \end{compactdesc}
\end{thm}

We can construct every graph from some its subgraphs via disjoint unions and operations of simple collapse.
This idea holds also for the cut ideal: The key is to use Theorem \ref{teo-elementaty-operation}.\textbf{5)} and Theorem \ref{thm-collapsing-generic-form} as we show in the following example.
In contrast, in Example \ref{ex-generators}, we stress that Theorem \ref{thm-collapsing-generic-form} does not allow to control the gene\-rators of the collapsing cut ideal since the ones of $I_G$.

\begin{ex}\em\label{ex-4-seg-con}
We compute $I_{P_4}$ via the cut ideal of $P_3 \sqcup K_2$, using a suitable vertices collapse.
We label the vertices of $K_2$ with $4$ and $5$. 
We saw in Example \ref{ex-G-3-SEG} that $I_{P_3}=(r_{1|23}r_{3|12} - r_{2|13} r_{123|\cdot})$ and also that $I_{K_2}=(0)$. 
We know how to produce the linear relation between the variables (like $r_{a\times b}-r_{a\times b^*}$). 

Let us focus on the non linear part. 
We need to start from an element in $I_{P_3}$: for instance we have $r_{2|13}r_{3|12}-r_{2|13}r_{3|12}=0\in I_{P_3}$. 
and $r_{4|5}r_{45|\cdot}-r_{4|5}r_{45|\cdot}=0\in I_{K_2}$. 
Thus we compose them into
\[
  r_{2\mathbf{4}|13\mathbf{5}}r_{3\mathbf{54}|12}-r_{2|13\mathbf{54}}r_{3\mathbf{5}|12\mathbf{4}}
\]
obtaining an element of $I_{P_3 \sqcup K_2}$.
In similar way, by changing only the element in $I_{P_3}$, one obtains also:
\begin{eqnarray*}
  r_{\mathbf{4}|123\mathbf{5}}r_{12|3\mathbf{45}}&-&r_{123\mathbf{45}|}r_{3\mathbf{5}|12\mathbf{4}},\\
  r_{1|23\mathbf{54}}r_{3\mathbf{5}|12\mathbf{4}}&-&r_{12|3\mathbf{54}}r_{1\mathbf{4}|23\mathbf{5}},\\
  r_{1|23\mathbf{54}}r_{13\mathbf{5}|2\mathbf{4}}&-&r_{2|13\mathbf{54}}r_{1\mathbf{4}|23\mathbf{5}},\\
  r_{1|23\mathbf{54}}r_{\mathbf{4}|123\mathbf{5}}&-&r_{123\mathbf{54}|}r_{1\mathbf{4}|23\mathbf{5}}.
\end{eqnarray*}

Moreover, considering the non zero generator $r_{1|23}r_{3|12} - r_{123|\cdot}r_{2|13}$ of $I_{P_3}$ and  $r_{45|\cdot}r_{45|\cdot}-r_{45|\cdot}r_{45|\cdot}\in I_{K_2}$ one has the following element of $I_{P_3 \sqcup K_2}$:
\[
  r_{1|23\mathbf{45}}r_{3\mathbf{45}|12}-r_{123\mathbf{45}|\cdot}r_{2|13\mathbf{45}}.
\]
Thus, by changing the element in $I_{K_2}$, one has:
\begin{eqnarray*}
 r_{\mathbf{4}|123\mathbf{5}}r_{13\mathbf{5}|2\mathbf{4}}&-&r_{1\mathbf{4}|23\mathbf{5}}r_{3\mathbf{5}|12\mathbf{4}},\\
 r_{1|23\mathbf{45}}r_{3\mathbf{5}|124}&-&r_{\mathbf{4}|123\mathbf{5}}r_{2|13\mathbf{45}},\\	  
 r_{1|23\mathbf{45}}r_{3\mathbf{5}|124}&-&r_{123\mathbf{45}|\cdot}r_{13\mathbf{5}|2\mathbf{4}}.
\end{eqnarray*}
This completes the non linear generators of the cut ideal of $I_{P_3 \sqcup K_2}$. 

We compute $I_{P_4}$ using Theorem \ref{thm-collapsing-generic-form}: one collapses the vertices $3$ and $5$. 
Thus one has:
\[
    I_{P_4}=\left(\begin{aligned}
		    r_{13|24}r_{12|34}-r_{2|134}r_{3|124},\\
		    r_{4|123}r_{12|34}-r_{1234|}r_{3|124},\\
		    r_{1|234}r_{3|124}-r_{12|34}r_{14|23},\\
		    r_{1|234}r_{13|24}-r_{2|134}r_{14|23},\\
		    r_{1|234}r_{4|123}-r_{1234|}r_{14|23},\\
		    r_{4|123}r_{13|24}-r_{14|23}r_{3|124},\\
		    r_{1|234}r_{3|124}-r_{4|123}r_{2|134},\\
		    r_{1|234}r_{3|124}-r_{13|24}r_{1234|\cdot},\\
		    r_{1|234}r_{12|34}-r_{1234|}r_{2|134}
                 \end{aligned}\right).
\]
\end{ex}

\begin{ex}\em\label{ex-generators}
Let 
\[
  G=(\{1,2,3,4,5\},\{\{1,2\},\{1,3\},\{2,3\},\{2,4\},\{3,4\},\{4,5\}\}).
\]
Its cut ideal is
\[
  I_{G}=\left(\begin{aligned}
		    r_{13|245}r_{4|1235}-r_{123|45}r_{24|135},	r_{12|345}r_{24|135}-r_{13|245}r_{34|125},\\
		    r_{12|345}r_{4|1235}-r_{123|45}r_{34|125},	r_{12|345}r_{13|245}-r_{2|1345}r_{3|1245},\\	
		    r_{124,35}r_{2|1345}-r_{12|345}r_{24|135},	r_{124,35}r_{123|45}-r_{3|1245}r_{4|1235},\\	
		    r_{124,35}r_{13|245}-r_{3|1245}r_{24|135},	r_{124,35}r_{12|345}-r_{3|1245}r_{34|125},\\	
		    r_{14|235}r_{3|1245}-r_{124,35}r_{23|145},	r_{14|235}r_{123|45}-r_{23|145}r_{4|1235},\\	
		    r_{14|235}r_{13|245}-r_{23|145}r_{24|135},	r_{14|235}r_{12|345}-r_{23|145}r_{34|125},\\	
		    r_{134|25}r_{23|145}-r_{14|235}r_{2|1345},	r_{134|25}r_{3|1245}-r_{12|345}r_{24|135},\\	
		    r_{134|25}r_{123|45}-r_{2|1345}r_{4|1235},	r_{134|25}r_{13|245}-r_{2|1345}r_{24|135},\\	
		    r_{134|25}r_{12|345}-r_{2|1345}r_{34|125},	r_{134|25}r_{124,35}-r_{34|125}r_{24|135},\\	
		    r_{5|1234}r_{2|1345}-r_{134|25}r_{12345|\cdot},	r_{5|1234}r_{3|1245}-r_{124,35}r_{12345|\cdot},	\\
		    r_{5|1234}r_{123|45}-r_{12345|\cdot}r_{4|1235},	r_{5|1234}r_{13|245}-r_{12345|\cdot}r_{24|135},	\\
		    r_{5|1234}r_{12|345}-r_{12345|\cdot}r_{34|125},	r_{5|1234}r_{14|235}-r_{15|234}r_{4|1235}, \\
		    r_{1|2345}r_{4|1235}-r_{5|1234}r_{23|145},	r_{1|2345}r_{4|1235}-r_{14|235}r_{12345|\cdot},	\\
		    r_{1|2345}r_{34|125}-r_{12|345}r_{15|234},	r_{1|2345}r_{24|135}-r_{13|245}r_{15|234},	\\
		    r_{1|2345}r_{4|1235}-r_{123|45}r_{15|234},	r_{1|2345}r_{123|45}-r_{12345|\cdot}r_{23|145},	\\
		    r_{1|2345}r_{124,35}-r_{3|1245}r_{15|234},	r_{1|2345}r_{14|235}-r_{23|145}r_{15|234},	\\
		    r_{1|2345}r_{134|25}-r_{2|1345}r_{15|234},	r_{1|2345}r_{5|1234}-r_{12345|\cdot}r_{15|234}	\\
                 \end{aligned}\right).
\]
We collapse the vertices $1$ and $5$ obtaining $K_4$. The cut ideal of $K_4$ is
\[
  I_{K_4}= (r_{1|234}r_{2|134}r_{3|124}r_{4|123}-r_{1234|\cdot}r_{23|14}r_{12|34}r_{13|24})
\]
We want to compute it using the Theorem \ref{thm-collapsing-generic-form}. 
We observe that all the generators of $I_G$ contain at least one of the non feasible variables $r_{1|2345}$, $r_{5|1234}$, $r_{134|25}$, $r_{14|235}$, $r_{124,35}$, $r_{12|345}$, $r_{13|245}$ and $r_{123|45}$; hence all of them will be killed. 
One has that $r_{15|234}r_{2|1345}r_{3|1245}r_{4|1235}-r_{12345|\cdot}r_{23|145}r_{125|34}r_{135|24}$ is in the cut ideal $I_G$; this element survives because contains only feasible elements; moreover the collapsing substitution produces exactly the generator we wanted.
\end{ex}

The rest of this section is devoted to the prove of Theorem \ref{thm-collapsing-generic-form}. 
The simple and singular cases are different so we split the proof in two proposition analysing them separately.

\subsection{The Simple Collapse}
In this section we study the simple collapse of graphs in the classical and non-classical case. The simple collapse does not change the number of edges or the multiplicities of them and if we start from a graph with trivial multiplicity, then we obtain a graph with trivial multiplicity. 

$G\#_0 H=(G\sqcup H)_{k\equiv k+1}$ is an example of simple collapse where the cut varieties are isomorphic. 
This is not a general fact:
\begin{ex}\em\label{ex-collapsing-tria}
	$K_3$ could be seen as the collapse of $1$ and $4$ in $P_4$.
	We compute that
	\begin{eqnarray*}
	 	R_{P_4}&=&\frac{\mathbb{K}[r_{1|234}, r_{2|134}, r_{3|124}, r_{4|123}, r_{12|34}, r_{13|24}, r_{14|23}, r_{1234|\cdot}]}{I_{P_4}},\\
		R_{K_3}&=&\frac{\mathbb{K}[r_{1|23}, r_{2|13}, r_{3|12}, r_{123|\cdot}]}{(0)},
	\end{eqnarray*}
	where $I_{P_4}$ is generated by the nine quadratic equations in Example \ref{ex-4-seg-con}.
	Hence $\Aff{P_4} \not\cong \Aff{(P_4)_{(1\equiv 4)}}$.
\end{ex}

The matrix $\mathcal{A}_{G}$ implicitly gives an order of the variables of $R_n$. In what follow we use the letter $p$ to denote the partition of a variable $r_p$ and the letter $k$ to indicate that $r_k$ is the $k$-th variable in this order. 

\begin{lem}\label{lemma-simple-collapse}
	Let $G$ be a graph and let $\{1, k\}\notin E$. Let the collapse of $1$ and $k$ be simple. 
	Then there exist a finite number of matrices $C_{1\equiv k}$ so that  $\mathcal{A}_G C_{1\equiv k}=\mathcal{A}_{G_{1\equiv k}}$.
	If $\mathcal{A}_G$ is made of the block matrices $(F, N)$ where $F$ (resp. $N$) is the matrix of the exponents of the image of the feasible (resp. non feasible) variables, then $C_{1\equiv k}$ is a $2^{n-1} \times 2^{n-2}$ matrix and it has the following block form
	\[
	 	C_{1\equiv k}=\left(\begin{array}{c}
	 	                     \operatorname{id}\\
				     0
	 	                    \end{array}\right) 
	\]
\end{lem}
\begin{proof}
	By the assumption, $\{1, k\}$ is not an edge, so $T_{A_G}=T_{A_{G_{1\equiv k}}}$. 
	Fixing the collapse we fix the injection $\tilde{(.)}:R_{n-1}\hookrightarrow R_n$ where $p$ is sent to $p_f$.
	The collapsing is simple and so one has 
	\begin{equation*}\label{eq-phi-map}
	    \phi_{G_{1\equiv k}}(r_{p})=\phi_{G}(r_{p_f})=\phi_{G}(\tilde{r_{p}}).
	\end{equation*}
	Thus $\phi_{G_{1\equiv k}}$ factors through the composition $R_{n-1}\stackrel{\tilde{(.)}}{\hookrightarrow}R_n\stackrel{\phi_{G}}{\rightarrow}T_{A_G}$.
	The injection map $\tilde{(.)}$ corresponds to the matrix $C_{1\equiv k}$.
\end{proof}

\begin{pro}\label{pro-collapse-simple}
Theorem \ref{thm-collapsing-generic-form} holds for simple collapses.
\end{pro}
\begin{proof}
	Using the previous lemma we know that $\phi_{G_{1\equiv k}}$ factors through the composition $R_{n-1}\stackrel{\tilde{(.)}}{\hookrightarrow}R_n\stackrel{\phi_{G}}{\rightarrow}T_{A_G}$.
	If $x\in R_{G_{1\equiv k}}$ and  $\phi_{G_{1\equiv k}}(x)=0$ then $\phi_{G}(\tilde{x})=0$, where $\tilde{x}$ is the lifting of $x$ in $R_n$. 
	This prove the $(substitute)$ rule. 
	The $(kill)$ property follows form the fact that $p_{nf}$ is a lifting that does not correspond to any partition in $G_{1\equiv k}$.
\end{proof}

\subsection{The Singular Collapse}
The notions of multiplicity and labeling that we introduced deal with the singular collapse.

We note that before and after a singular collapse the domain changes because the number of vertices change as well; the codomain changes because we decrease the number of the edges decrease. 

\begin{lem}\label{lemma-singular-easy-collapse}
	Let $G$ be a graph and let $\{1, k\}\notin E$. Let the collapse of $1$ and $k$ be singular and let the collapsing pairs of edges have the same labels for each pair. 
	Then there exist a finite number of matrices $C_{1\equiv k}$ so that $\mathcal{A}_G C_{1\equiv k}=\mathcal{A}_{G_{1\equiv k}}$ with the same form of Lemma \ref{lemma-simple-collapse}.
\end{lem}
\begin{proof}
	The proof follows as in Lemma \ref{lemma-simple-collapse}. (\ref{eq-phi-map}) holds because the feasible partitions separate or divide, at the same time, the collapsing edges.
\end{proof}

\begin{lem}\label{lemma-singular-collapse}
	Let $G$ be a graph and let $\{1, k\}\notin E$. Let the collapse of $1$ and $k$ be singular. 
	Then there are a finite number or matrices $C_{1\equiv k}$ such that  $(R_{l'} \mathcal{A}_G) C_{1\equiv k} =\mathcal{A}_{G_{1\equiv k}}$, where
	\begin{compactitem}
	    \item $l'$ is the labeling such that each collapsing couple of edges has the same labels;
	    \item $C_{1\equiv k}$ is the matrix of the collapse as in the previous lemma.
	\end{compactitem}
\end{lem}
\begin{proof}
	Without loss of generality we can assume that the singular collapse involves only two edges: $\{1, l\}$ and $\{l, k\}$. 
	Any singular collapse splits in two steps.
	We work in the non-classical setting, so it is possible that $\{1, l\}$ and $\{l, k\}$ have the same label; if not, we relabel them with the same one, $a$. We call the new label $l'$.
	This gives a labelled graph $(G, l', A', \sigma)$. Using \textbf{iii.bis)} of Lemma \ref{lem-elementary-operation} to the relabeling correspond a unique matrix $R_{l'}$.
	We collapse the two vertices of $(G, l', A', \sigma)$. 
	Using the previous lemma, there is a a matrix $C_{1\equiv k}$ controlling the collapse.
\end{proof}
In other words, before we collapse the two vertices in the codomain, that is we let the two edges ($\{1, l\}$, $\{l, k\}$) be considered as a unique edge in $T_A$; then we collapse the two vertices in the domain, that is we select the feasible partition of $[n-1]$.
\begin{pro}\label{pro-collapse-singular}
	Theorem \ref{thm-collapsing-generic-form} holds for singular collapse. 
\end{pro}
\begin{proof}
	Without lost of generality we assume that the singular collapse involves only two edges: $\{1, l\}$ and $\{l, k\}$. 
	Using the previous lemma we know that $\phi_{(G_{1\equiv k}, l, A, \sigma)}$ factors through
	\[
	    R_{n-1}\stackrel{\tilde{(.)}}{\hookrightarrow}  R_{n} \stackrel{\phi_{G}}{\rightarrow} T_{G} \twoheadrightarrow T_{A_{G_{1\equiv k}}}.
	\]
	Let $a$ be the label of the two collapsing edges and let $l'$ be the new labeling. 
	If $r_{q}\in R_{n}$, looking at the map $\phi_{G}$, one sees that
	\begin{equation}\label{eq-phi-possibility}
	    \phi_{(G,l',A')}(r_{q})=\begin{cases}
			\cdots s_{a}^{\sigma(1,l)}t_{a}^{\sigma(l,k)} \cdots	&\text{ $q$ sep. $\{1,l\}$ but not $\{l,k \}$;}\\
			\cdots s_{a}^{\sigma(l,k)}t_{a}^{\sigma(1,l)} \cdots	&\text{ $q$ sep. $\{l,k \}$ but not $\{1,l\}$;}\\
			\cdots s_{a}^{\sigma(1,l)+\sigma(l,k)} \cdots 		&\text{ $q$ sep. both $\{1,l\}, \{l,k \}$;}\\
			\cdots t_{a}^{\sigma(1,l)+\sigma(l,k)} \cdots 		&\text{ otherwise. }
				      \end{cases}
	\end{equation}
	After the collapsing of $1$ and $k$, following the notation of the previous lemma, we get the graph $(G_{1\equiv k}, l', A', \sigma')$. The collapse produces a change of the multiplicity of the edge $\{1,l\}$: $\sigma'(1,l)=\sigma(1,l)+\sigma(l,k)$. 
	Let $r_{p}\in R_{n-1}$, one has 
	\[
	    \phi_{(G_{1\equiv k}, l', A', \sigma)}(r_{p})=\begin{cases}
			\cdots s_{a}^{\sigma(1,l)+\sigma(l,k)} \cdots &\text{ $p$ separates the vertices $1$ and $l$;}\\
			\cdots t_{a}^{\sigma(1,l)+\sigma(l,k)} \cdots &\text{ otherwise.}
							\end{cases}
	\]
	The maps $\phi_{(G, l', A', \sigma)}$ and $\phi_{(G_{1\equiv k},l', A', \sigma')}$ are coherent: if $p=A|B$ is a $(n-1)$-partition, then $\phi_{(G_{1\equiv k}, \sigma)}(r_{p})=\phi_{(G, l, A)}(\tilde{r_{p}})$, where $\tilde{r_{p}}$ is the usual lifting of $r_p$. 
	There are no partitions $p$ of $[n-1]$ such that $p'$ separates only one of the edges $\{1,l\}, \{l,k \}$: this implies that the first case of the equation (\ref{eq-phi-possibility}) is not possible after the collapsing of $1$ and $k$.
\end{proof}

\section*{Acknowledgements}
I thank Ralf Fr\"{o}berg for the incredible help during these last nine months. 
I thank Alexander Engstr\"{o}m for introducing me to this topics and for finding Example \ref{ex-generators}.
This article would not have this short look without the help and the corrections of Brun\"{o} Benedetti.

Finally, this work began during the Summer School Pragmatic 2011. I thank again the organizers.

\addcontentsline{toc}{section}{
\end{document}
\textbf{Bibliography}}
\begin{thebibliography}{1}

\bibitem{Deza-Laurent}
Michel~Marie Deza and Monique Laurent.
\newblock {\em Geometry of cuts and metrics}, volume~15 of {\em Algorithms and
  Combinatorics}.
\newblock Springer-Verlag, Berlin, 1997.

\bibitem{Diaconis-Sturmfels}
Persi Diaconis and Bernd Sturmfels.
\newblock Algebraic algorithms for sampling from conditional distributions.
\newblock {\em Ann. Statist.}, 26(1):363--397, 1998.

\bibitem{engstrom}
Alexander Engstr\"{o}m.
\newblock Cut ideals of $k_4$-minor free graphs are generated by quadrics.
\newblock Michigan Math. J. 60 (2011), no. 3, 705-714.

\bibitem{Nagee-Petro}
Uwe Nagel and Sonja Petrovi{\'c}.
\newblock Properties of cut ideals associated to ring graphs.
\newblock {\em J. Commut. Algebra}, 1(3):547--565, 2009.

\bibitem{oltreanu}
Anda Olteanu.
\newblock Monomial cut ideals.
\newblock arXiv:1105.3564v1, 2011.

\bibitem{Sottile_toric}
Frank Sottile.
\newblock Toric ideals, real toric varieties, and the moment map.
\newblock In {\em Topics in algebraic geometry and geometric modeling}, volume
  334 of {\em Contemp. Math.}, pages 225--240. Amer. Math. Soc., Providence,
  RI, 2003.

\bibitem{Sturmfels-Sullivant}
Bernd Sturmfels and Seth Sullivant.
\newblock Toric geometry of cuts and splits.
\newblock {\em Michigan Math. J.}, 57:689--709, 2008.
\newblock Special volume in honor of Melvin Hochster.

\end{thebibliography}
